\documentclass[a4paper,12pt]{article}
\usepackage{amsthm}
\usepackage{amssymb}
\usepackage{amsmath}
\usepackage{amsfonts}
\usepackage{color}
\usepackage{graphicx}

\newtheorem{remark}{Remark}[section]

\newtheorem{theorem}{Theorem}[section]

\def\b1{\mbox{\boldmath $1$}}

\newenvironment{demo*}{\vspace{3mm}\noindent{\bf Proof.}}{\hfill $\Box$ \vspace{3mm}}

\begin{document}

\title{\bf \Large {\bf   Absolute moments of GH distribution and zeros of zeta function }}
{\author{\normalsize{Chuancun Yin}\\{\normalsize\it    (School of Statistics and Data Science, Qufu Normal University}\\
\noindent{\normalsize\it Shandong 273165, China)}}}
\maketitle


\vskip0.01cm

\noindent{\large {\bf Abstract}} In this work, we focus on  deriving the absolute moments of the generalized hyperbolic (GH) distribution via its mean-variance mixture structure. Conditioning on the GIG mixing variable reduces the problem to Gaussian moment calculations and known GIG power moments, avoiding direct integration with Bessel functions.
The variance gamma (VG), normal inverse Gaussian (NIG), hyperbolic,
Normal, and Student's $t$ distributions
 emerge as  natural special  or limiting cases of the generalized hyperbolic family. This approach offers a substantially simpler and more transparent derivation, facilitating both implementation and verification.

\medskip

\noindent{\bf Keywords:}  {\rm  { Absolute moment, Generalized hyperbolic distribution,  Normal inverse gaussian  distribution, Raw moment,   Variance-gamma distribution} }

\noindent{\bf 2020 MSC}: 92D25, 37H15, 60H10, 60J60


\numberwithin{equation}{section}
\section{Introduction}\label{intro}
The problem of obtaining explicit expressions for absolute moments is a classical and well-motivated one in distribution theory.
Classical results exist for the normal distribution (Winkelbauer \cite{W2012}), the generalized hyperbolic (GH) family (Barndorff-Nielsen \& Stelzer, 2005),
 and the multivariate normal distribution, among others. In particular, Barndorff-Nielsen and Stelzer (2005) made two main contributions:
 (i) they derived explicit formulae for the absolute moments of the GH family and its important subclass, the normal inverse Gaussian (NIG) distribution,
 with expressions given in terms of series of Bessel functions and linked to the moments of their symmetric (zero-skewness) counterparts;
  (ii) they investigated the scaling behaviour of NIG Levy processes by computing the logarithmic derivative of centred absolute moments with respect to time,
  revealing that these processes exhibit approximate rather than exact scaling and providing insights into this phenomenon.
   Subsequently, Gaunt (2023) obtained new closed-form formulas for the moments and absolute moments of the variance-gamma (VG) distribution,
   which in turn led to new formulas for the moments and absolute moments of the product of two correlated zero-mean normal random variables.
   The variance-gamma distribution is a flexible four-parameter family widely used in finance, insurance, and statistics.
    Although its ordinary moments are well known, explicit expressions for its absolute moments-which are crucial for Markov inequalities,
     Berry-Esseen type bounds, and scaling analyses had been lacking for the fully asymmetric case. Gaunt (2025) derives the
closed-form expressions for the absolute raw and central moments of the four-parameter Variance-Gamma (VG) distribution, $\mathbb{E}[|X|^{r}]$
and $\mathbb{E}\left[|X-\mathbb{E}[X]|^{r}\right]$, valid for the full parameter space and non-zero location parameters.
Its results are not only of theoretical interest but also offer practical tools for estimation, approximation, and further probabilistic modelling,
and the identified future directions show that this work serves as a solid foundation for ongoing research in distribution theory and its applications.
  Scott et al. (2011) provided the first computationally efficient recursive framework for calculating the moments of the generalized hyperbolic (GH) distribution,
expressing them as finite sums of modified Bessel function ratios. This advancement resolved the limitations of earlier integral- based methods
and enabled stable computation of higher-order statistics across the entire GH family. In this paper, we leverage their recursive technique
to derive the absolute raw and central moments of the GH distribution, extending the existing toolkit to support robust risk management applications.


Nevertheless, the existing closed-form expressions for the generalized hyperbolic class, while mathematically rigorous, are often encumbered by the analytic complexity of their Bessel-function series representations. Such series, though explicit, can be computationally demanding for high-order moments or specific parameter regimes, and they do not always readily reveal the role of the individual shape, skewness, and scaling parameters in determining the moment structure. Moreover, the established literature has primarily focused either on symmetric subclasses or on particular limiting cases, leaving the full generality of the GH family  particularly with respect to both raw and central absolute moments under arbitrary parametrisation-less systematically addressed in a unified framework.

The purpose of this paper is to give a unified formula for absolute moments of the GH distribution, including absolute raw moments and absolute central moments, with real-valued orders, based on its mean-variance Gaussian mixture representation. This hierarchical perspective, which interprets the GH distribution as a normal variance-mean mixture driven by a generalised inverse Gaussian (GIG) mixing distribution, offers a twofold advantage. First, it reduces the problem of moment calculation to that of evaluating moments of the GIG mixing variable, for which explicit and well-structured formulae are available in terms of modified Bessel functions of the second kind. Second, and more importantly, this representation yields moment expressions that are not only fully explicit but also structurally transparent: they separate the contributions of the location, skewness, and tail-weight parameters in a tractable manner. As a result, our derivation provides a direct probabilistic interpretation of each moment component, facilitates efficient numerical implementation across the entire GH parameter space, and generalises several earlier results as special cases. Furthermore, by encompassing both raw and central absolute moments in a coherent framework, our results offer a more comprehensive toolkit for moment-based inference, tail-risk quantification, and the refined scaling analysis of GH-driven Levy processes beyond the NIG subclass.

The remainder of this paper is organised as follows. Section 2 provides preliminaries, reviewing the GH distribution and its GIG mixture representation and setting up the basic notation. Section 3 presents the main theoretical results, namely the unified closed-form expressions for both raw and central absolute moments. Section 4 discusses  some special cases. Section 5 concludes the paper.

{\bf Notation}. For later use, we will require two common special functions.  The Kummer's confluent
hypergeometric function is
\[
{}_1F_1(\alpha, \gamma; z) = \sum_{n=0}^{\infty} \frac{(\alpha)_n}{(\gamma)_n} \, \frac{z^n}{n!},
\]
the Gaussian hypergeometric function is
\[
{}_2F_1(a,b;c;x)
=\sum_{n=0}^{\infty}
\frac{(a)_n\,(b)_n}{(c)_n}\,\frac{x^n}{n!},
\]
where $(q)_n=\Gamma(q+n)/\Gamma(q)$ denotes the Pochhammer symbol (rising factorial).
The modified Bessel function of the second kind is
\[
K_\nu(x)=\frac{\pi}{2}\,
\frac{I_{-\nu}(x)-I_\nu(x)}{\sin(\nu\pi)},
\qquad
I_\nu(x)=\sum_{n=0}^\infty
\frac{1}{n!\,\Gamma(n+\nu+1)}
\left(\frac{x}{2}\right)^{2n+\nu}.
\]
An integral representation is
\[
K_\nu(x)=\int_0^\infty e^{-x\cosh t}\cosh(\nu t)\,dt,
\quad x>0.
\]
Asymptotically,
\[
K_\nu(x)\sim
\frac12\Gamma(\nu)\Bigl(\frac{x}{2}\Bigr)^{-\nu}
\;(x\to0^+,\,\Re\nu>0),\qquad
K_\nu(x)\sim\sqrt{\frac{\pi}{2x}}\,e^{-x}
\;(x\to\infty).
\]

The generalized Laguerre polynomials are defined by
\[
L_n^{(\alpha)}(x)=\binom{n+\alpha}{n}\,{}_1F_1(-n,\alpha+1;x)
=\frac{(\alpha+1)_n}{n!}\,{}_1F_1(-n,\alpha+1;x),
\qquad \alpha>-1.
\]
Equivalently,
\[
L_n^{(\alpha)}(x)=\sum_{k=0}^{n}
(-1)^k\binom{n+\alpha}{n-k}\frac{x^k}{k!}.
\]
Rodrigues' formula reads
\[
L_n^{(\alpha)}(x)=\frac{x^{-\alpha}e^{x}}{n!}
\frac{d^n}{dx^n}\!\bigl(e^{-x}x^{n+\alpha}\bigr).
\]

\section{Generalized hyperbolic distribution}

 This section provides preliminaries, reviewing the GH distribution and its    limiting or particular cases.
 The generalized hyperbolic (GH) distribution, introduced by Barndorff-Nielsen (1977, 1978), is a flexible family of distributions that arises as a normal variance--mean mixture with a generalized inverse Gaussian (GIG) mixing variable. 
A random variable $X$   from this family has
 the density
\begin{equation}
f_X(x) = \frac{(\gamma/\delta)^\lambda}{\sigma \sqrt{2\pi} K_\lambda(\delta\gamma)}
\cdot
\frac{K_{\lambda - 1/2}\left( \frac{\alpha}{\sigma} \sqrt{\delta^2 \sigma^2 + (x - \mu)^2} \right)}
{\left( \frac{\sqrt{\delta^2 \sigma^2 + (x - \mu)^2}}{\sigma\alpha} \right)^{1/2 - \lambda}}
\cdot e^{\frac{\beta (x - \mu)}{\sigma^2}},\, x\in\mathbb{R},
\end{equation}
where $\chi = \delta^2$, $\psi = \gamma^2 = \alpha^2 - \beta^2$,  and $K_\lambda$ is the modified Bessel function of the second kind. Here, the parameter constraints are
 $\alpha > 0$, $|\beta| < \alpha$, and $\delta > 0$. Under this parameterization, the GH distribution is denoted by
\[
X \sim \mathrm{GH}(\lambda, \alpha, \beta, \delta, \mu, \sigma).
\]
The characteristic function of the GH distribution   is given by
\begin{equation}
\phi_X(t) = e^{it\mu} \left(\frac{\alpha^2 - \beta^2}{\alpha^2 - (\beta + i\sigma t)^2}\right)^{\lambda/2}
\cdot
\frac{K_\lambda\left(\delta \sqrt{\alpha^2 - (\beta + i\sigma t)^2}\right)}
{K_\lambda\left(\delta \sqrt{\alpha^2 - \beta^2}\right)}.
\end{equation}
 When $\sigma = 1$, we recover the standard GH parameterization $\mathrm{GH}(\lambda, \alpha, \beta, \delta, \mu)$.  For the research background and the latest advances on generalized hyperbolic (GH) distributions, see  Jiang, Nadarajah and Hitchen (2024), which systematically surveys the GH family.  For a comprehensive account of the Variance-Gamma distribution, its distributional theory, estimation methods and applications, we refer to Fischer, Gaunt and Sarantsev (2025).

The GH distribution contains many known distributions as limiting or particular cases. These include the variance gamma (VG), normal inverse Gaussian (NIG), hyperbolic,  Normal, and Student's $t$ distributions. Its flexibility arises from the parameters $\lambda$, $\alpha$, $\beta$, $\delta$, and $\sigma$, which jointly control tail heaviness, skewness, kurtosis, and scale. This makes the GH distribution a powerful tool in finance, physics, geology, and statistics.
 See, e.g., Barndorff-Nielsen \& Bl${\ae}$sild (1981) and Bl${\ae}$sild \& Jensen (1981), Prause (1999), Eberlein (2001),
Eberlein \& Hammerstein (2004) and Mencia \& Sentana (2004). Some recent results can be found in Barndorff-Nielsen and Stelzer (2005), Scott et al. (2011),   Gaunt (2023), Fischer,  Gaunt and Sarantsev (2025) and Gaunt (2025).

We now summarize important special cases obtained by fixing specific
parameter values or by considering limiting regimes.

\subsection{The Normal Inverse Gaussian  Distribution}

When \(\lambda=-\tfrac12\), the generalized hyperbolic distribution
reduces to the normal-inverse Gaussian (NIG) distribution (see, e.g. Paolella (2007)).
Using the identity \(K_{-1/2}(x)=K_{1/2}(x)=\sqrt{\frac{\pi}{2x}}e^{-x}\), the CF reduces to
\begin{equation*}
\phi_X(t)=e^{it\mu}\exp\left\{-\delta\left(\sqrt{\gamma^2+\sigma^2 t^2-2it\beta}-\gamma\right)\right\}.
\end{equation*}
The corresponding probability density function is
\begin{equation}
f_X(x) =
\frac{\delta \alpha}{\pi} \,
\frac{
\exp\!\left\{\delta \gamma + \frac{\beta (x-\mu)}{\sigma^2}\right\}
}{
\sqrt{\delta^2 \sigma^2 + (x-\mu)^2}
} \,
K_1\!\left(
\frac{\alpha}{\sigma} \sqrt{\delta^2 \sigma^2 + (x-\mu)^2}
\right).
\end{equation}

This density is valid for all \(x \in \mathbb{R}\), with parameters \(\delta > 0\), \(\sigma > 0\), \(\gamma > 0\), \(\beta \in \mathbb{R}\), and \(\alpha = \sqrt{\gamma^2 + \beta^2} > 0\).

\subsection{The Hyperbolic  Distribution}

Setting \(\lambda=1\) yields the hyperbolic distribution,
which arises from a GIG mixing law with shape parameter \(\lambda=1\)
and was originally introduced by Barndorff-Nielsen (1977).
  The CF of $X$ reads
\begin{equation*}
\phi_X(t)=e^{it\mu}\,\frac{\gamma}{\sqrt{\gamma^2+\sigma^2 t^2-2it\beta}}
\frac{K_1\!\left(\delta\sqrt{\gamma^2+\sigma^2 t^2-2it\beta}\right)}{K_1(\delta\gamma)},
\end{equation*}
where   \(K_\nu(\cdot)\) is the modified Bessel function of the third kind, and the parameters satisfy \(\delta>0\), \(\gamma>0\), \(\beta\in\mathbb{R}\), \(\sigma>0\).

The corresponding probability density function is:
\begin{equation}
f_X(x)
=
\frac{\gamma}{2\,\delta\,\sigma\,\alpha\, K_1(\delta\gamma)}
\exp\left\{
-\frac{\alpha}{\sigma} \sqrt{\delta^2 \sigma^2 + (x - \mu)^2}
+ \frac{\beta (x - \mu)}{\sigma^2}
\right\},
\qquad x \in \mathbb{R},
\end{equation}
where $\alpha = \sqrt{\gamma^2 + \beta^2}$, \(\delta>0,\ \gamma>0,\ \sigma>0,\ \beta\in\mathbb{R}\).

\subsection{Symmetric GH}

Setting \(\beta = 0\) yields the symmetric GH distribution,
for which the mean-variance mixture simplifies to
\[
X = \mu + \sigma\sqrt{W}\,Z,
\]
so that \(X\) is symmetric about \(\mu\).  The CF becomes real and even:

\begin{equation}
\phi_X(t)=e^{it\mu}
\left(\frac{\alpha^2}{\alpha^2+\sigma^2 t^2}\right)^{\lambda/2}
\frac{K_\lambda\!\left(\delta\sqrt{\alpha^2+\sigma^2 t^2}\right)}{K_\lambda(\delta\alpha)}.
\end{equation}
The corresponding probability density function is
\begin{equation}
f_X(x)=
\frac{\alpha^{1/2}}{\sigma\sqrt{2\pi}\,\delta^\lambda K_\lambda(\delta\alpha)}
\left(
\sqrt{\delta^2+\frac{(x-\mu)^2}{\sigma^2}}
\right)^{\lambda-\frac12}
K_{\lambda-\frac12}\!
\left(
\alpha\sqrt{\delta^2+\frac{(x-\mu)^2}{\sigma^2}}
\right),
\qquad x\in\mathbb R
\end{equation}
where the parameters satisfy
\[
\alpha>0,\quad \delta>0,\quad \sigma>0,\quad \lambda\in\mathbb R,\quad \mu\in\mathbb R,
\]
and \(K_\nu(\cdot)\) denotes the modified Bessel function of the second kind.


\subsection{Variance-gamma  Distribution}

For \(\lambda>0\), the limit \(\delta\to0\) (i.e.\ \(\chi\to0\))
reduces the GH distribution to the variance-gamma (VG) distribution.
This corresponds to replacing the GIG mixing distribution by a
gamma distribution, so that \(X\) becomes a variance-mean mixture
of a normal distribution with a gamma-distributed mixing variable.

 Using the small-argument asymptotics \(K_\lambda(x)\sim \frac{\Gamma(\lambda)}{2}(x/2)^{-\lambda}\) as \(x\to0\), the CF converges to
\begin{equation}
\phi_X(t)\;\longrightarrow\; e^{it\mu}
\left(\frac{\gamma^2}{\gamma^2+\sigma^2 t^2-2it\sigma\beta}\right)^{\lambda}.
\end{equation}
This is the CF of the variance gamma distribution with shape parameter \(\lambda\).
Its probability density function is
 \begin{equation}
f_X(x) =
\frac{\gamma^{2\lambda}}{\sigma \sqrt{\pi} \, \Gamma(\lambda)}
\left( \frac{|x - \mu|}{\sqrt{2}} \right)^{\lambda - 1/2}
e^{\frac{\beta (x - \mu)}{\sigma^2}}
K_{\lambda - 1/2}\!
\left(
\frac{\sqrt{\gamma^2 \sigma^2 + \beta^2}}{\sigma^2} |x - \mu|
\right), x \in \mathbb{R}.
 \end{equation}
The parameters satisfy
\[
\gamma > 0,\quad \sigma > 0,\quad \lambda > 0,\quad \beta \in \mathbb{R},\quad \mu \in \mathbb{R},
\]
and \(K_\nu(\cdot)\) denotes the modified Bessel function of the second kind. See, e.g. Fischer et al.(2025), Gaunt (2025).

\subsection{Skewed \(t\)-Distribution}

Let \(\gamma\to0\) (equivalently \(\alpha\to|\beta|\)) and set \(\lambda=-\nu/2\). The limiting distribution is the generalized hyperbolic skewed \(t\).
The probability density function   is
\begin{equation}
f_X(x)
=
\frac{2^{1-\nu/2}}{\Gamma(\nu/2)}
\frac{\delta^{\nu} |\beta|^{(\nu+1)/2}}{\sqrt{2\pi}\,\sigma}
\frac{
K_{(\nu+1)/2}\!\left(|\beta|\sqrt{\delta^2 + \left(\frac{x-\mu}{\sigma}\right)^2}\right)
}{
\left(\sqrt{\delta^2 + \left(\frac{x-\mu}{\sigma}\right)^2}\right)^{(\nu+1)/2}
}
\exp\!\left(\frac{\beta (x-\mu)}{\sigma^2}\right)
\label{eq:skewtpdf}
\end{equation}
for all \(x\in\mathbb{R}\), with parameters:
\[
\nu>0,\quad \delta>0,\quad \beta\in\mathbb{R},\quad \sigma>0,\quad \mu\in\mathbb{R}.
\]
When $\beta=0$, it degenerates to  Student's $t$-distribution.



For quick reference, the main special cases are collected in Table~1.
\begin{table}[htbp]
\centering
\caption{Special cases of the GH characteristic function. Here \(z=\sqrt{\gamma^2+\sigma^2 t^2-2it\beta}\).}
\label{tab:special}
\begin{tabular}{|c|c|c|}
\hline
Condition & Distribution name & Characteristic function \\
\hline
\(\lambda=-\frac12\) & Normal Inverse Gaussian (NIG) & \(e^{it\mu}\exp\{-\delta(z-\gamma)\}\) \\
\(\lambda=\frac12\) & Simplified GH & \(e^{it\mu}\,\frac{\gamma}{z}\,e^{-\delta(z-\gamma)}\) \\
\(\lambda=1\) & Hyperbolic (HYP) & \(e^{it\mu}\,\frac{\gamma}{z}\,\frac{K_1(\delta z)}{K_1(\delta\gamma)}\) \\
\(\beta=0\) & Symmetric GH & \(e^{it\mu}\left(\frac{\alpha^2}{\alpha^2+\sigma^2 t^2}\right)^{\lambda/2}\frac{K_\lambda(\delta\sqrt{\alpha^2+\sigma^2 t^2})}{K_\lambda(\delta\alpha)}\) \\
\(\delta\to0\) & Variance Gamma (VG) & \(e^{it\mu}\left(\frac{\gamma^2}{z^2}\right)^{\lambda}\) \\
\(\gamma\to0,\ \lambda=-\nu/2\) & Skewed \(t\) & \(e^{it\mu}\,z^{\nu/2}K_{\nu/2}(\delta z)\) (unnormalised) \\
\(\delta,\gamma\to\infty\) & Normal & \(e^{it\mu}e^{-\frac12\sigma^2 t^2}\) \\
\hline
\end{tabular}
\end{table}

\section{Absolute moment of GH distribution}

 Motivated by the recent work of Kirkby, Nguyen and Nguyen (2025), who utilized a normal-gamma mixture representation to unify moment derivations for the Student's t-distribution, we adopt the same technique in this section to derive the absolute raw and central moments of the GH distribution.

It is straightforward to show that the random variable $X \sim \mathrm{GH}(\lambda, \alpha, \beta, \delta, \mu, \sigma)$
 can be represented as a mean-variance Gaussian mixture:
\begin{equation}
\label{eq:gh-def}
X = \mu + \beta W + \sigma \sqrt{W}\, Z,
\end{equation}
where $Z \sim \mathcal{N}(0,1)$, $W \sim \mathrm{GIG}(\lambda, \chi, \psi)$, a generalized inverse Gaussian distribution, $Z$ and $W$ are independent, $\mu \in \mathbb{R}$ is a location parameter,
      $\beta \in \mathbb{R}$ is a skewness parameter,
      $\sigma > 0$ is a scale parameter,
     $\lambda \in \mathbb{R}$ is a shape parameter,
      $\chi > 0$ and $\psi > 0$ are parameters of the GIG mixing distribution.
 The density of
\begin{equation*}
W \sim \operatorname{GIG}(\lambda, \chi, \psi),
\end{equation*}
is given by
\begin{equation}
f_W(w) = \frac{(\psi/\chi)^{\lambda/2} w^{\lambda-1}}{2K_\lambda(\sqrt{\chi\psi})} \exp\left(-\frac{1}{2}\left(\frac{\chi}{w} + \psi w\right)\right), w>0.
\end{equation}
The $p$-th moment of  $W$ is
(J$\phi$rgensen 1982, Ch.~1; Barndorff-Nielsen 1977)
\begin{equation}
\mathbb{E}[W^p]
=
\left(\frac{\chi}{\psi}\right)^{p/2}
\frac{K_{\lambda+p}(\sqrt{\chi\psi})}
{K_\lambda(\sqrt{\chi\psi})},
\qquad \lambda+p>0.
\end{equation}

\subsection{Absolute central moments}

\begin{theorem} For $r>\max(-1,-\lambda)$, the absolute moment of the GH distribution is
\begin{equation}
\begin{aligned}
\mathbb{E}[|X-\mu|^r]
= \;&
\sigma^{r} \frac{2^{r/2}\,\Gamma\!\left(\frac{r+1}{2}\right)}{\sqrt{\pi}}
\left(\frac{\chi}{\psi}\right)^{r/4}
\sum_{n=0}^{\infty}
\frac{(-r/2)_n}{(1/2)_n\, n!}
\left( -\frac{\beta^2}{2\sigma^2} \sqrt{\frac{\chi}{\psi}} \right)^n
\frac{K_{\lambda + n + r/2}(\sqrt{\chi\psi})}{K_\lambda(\sqrt{\chi\psi})},
\end{aligned}
\end{equation}
with \(\chi = \delta^2 > 0\), \(\psi = \gamma^2 = \alpha^2 - \beta^2 > 0\).
\end{theorem}
\begin{proof}
 Conditional on $W = w$,
\[
X \mid W = w \;\sim\; N\!\left(\mu + \beta w,\; \sigma^2 w\right),
\qquad
W \sim \operatorname{GIG}(\lambda, \chi, \psi),
\]
where the GIG density is (\(\chi>0,\;\psi>0\)):
\[
f_W(w) = \frac{(\psi/\chi)^{\lambda/2}}{2K_\lambda(\sqrt{\chi\psi})} w^{\lambda-1} \exp\left[-\frac{1}{2}\left(\frac{\chi}{w} + \psi w\right)\right], \quad w>0.
\]

Define $Z \sim \mathcal{N}(0, \sigma^{2})$ independently of $W$. Then
\[
X - \mu \;=\; \beta W \;+\; \sqrt{W}\,Z,
\qquad Z \perp W.
\]
Consequently, for any $r \in \mathbb{R}$ for which the expectation exists,
\[
|X - \mu|^{r}
\;=\; W^{r/2}\,|\beta\sqrt{W} + Z|^{r},
\]
and taking expectations gives
 \begin{equation}
\mathbb{E}\!\left[|X - \mu|^{r}\right]
\;=\;
\mathbb{E}\!\left[
W^{r/2}\,
\mathbb{E}\!\left[|\beta\sqrt{W} + Z|^{r} \;\middle|\; W\right]
\right].
 \end{equation}
We first evaluate the inner absolute moment of the normal variable.
For $Z \sim \mathcal{N}(0, \sigma^{2})$ and any constant $a \in \mathbb{R}$,
the absolute moment of $a + Z$ is known to be (see, e.g. Winkelbauer (2014))
 \begin{equation}
\mathbb{E}\!\left[|a + Z|^{r}\right]
\;=\;
\frac{\sigma^{r}\,2^{r/2}\,\Gamma\!\left(\frac{r+1}{2}\right)}{\sqrt{\pi}}\,
{}_1F_1\!\left(
-\frac{r}{2};\; \frac{1}{2};\; -\frac{a^{2}}{2\sigma^{2}}
\right),
\qquad r > -1,
 \end{equation}
where ${}_1F_1(\cdot;\cdot;\cdot)$ is Kummer's confluent hypergeometric function.
Substituting $a = \beta\sqrt{W}$ yields
 \begin{equation}
\mathbb{E}\!\left[|\beta\sqrt{W} + Z|^{r} \;\middle|\; W = w\right]
=
\sigma^{r}\,
\frac{2^{r/2}\,\Gamma\!\left(\frac{r+1}{2}\right)}{\sqrt{\pi}}\,
{}_1F_1\!\left(
-\frac{r}{2};\; \frac{1}{2};\; -\frac{\beta^{2} w}{2\sigma^{2}}
\right).
 \end{equation}
Inserting (3.5) into (3.3) produces the following closed-form formula
\begin{equation}
\mathbb{E}\!\left[|X - \mu|^{r}\right]
=
\sigma^{r}\,
\frac{2^{r/2}\,\Gamma\!\left(\frac{r+1}{2}\right)}{\sqrt{\pi}}\,
\mathbb{E}\!\left[
W^{\frac{r}{2}}\,
{}_1F_1\!\left(
-\frac{r}{2};\; \frac{1}{2};\; -\frac{\beta^{2} W}{2\sigma^{2}}
\right)
\right]
 \end{equation}
for $r > -1$ (or the appropriate range ensuring finiteness),  where \({}_1F_1\) is the confluent hypergeometric function (Kummer's function).

Expand the hypergeometric function as a power series:
\[
{}_1F_1\!\left(-\frac{r}{2}; \frac{1}{2}; -\frac{\beta^{2} w}{2\sigma^{2}}\right)
= \sum_{n=0}^{\infty} \frac{(-r/2)_n}{(1/2)_n\, n!}
\left(-\frac{\beta^{2}}{2\sigma^{2}}\right)^n w^n,
\]
where \((a)_n = a(a+1)\cdots(a+n-1)\) denotes the Pochhammer symbol.

Substituting into the expectation and interchanging sum and integral:
\[
\mathbb{E}\!\left[ W^{r/2} {}_1F_1(\cdots) \right]
= \sum_{n=0}^{\infty} \frac{(-r/2)_n}{(1/2)_n\, n!}
\left(-\frac{\beta^{2}}{2\sigma^{2}}\right)^n \mathbb{E}(W^{n+r/2}).
\]
Bying using (3.3),
\[
 \mathbb{E}(W^{n+r/2})
= \left(\frac{\chi}{\psi}\right)^{n/2 + r/4}
\frac{K_{\lambda+n+r/2}(\sqrt{\chi\psi})}{K_\lambda(\sqrt{\chi\psi})}.
\]
Substituting back into the expectation:
\[
\mathbb{E}\!\left[ W^{r/2} {}_1F_1(\cdots) \right]
= \left(\frac{\chi}{\psi}\right)^{r/4}
\sum_{n=0}^{\infty}
\frac{(-r/2)_n}{(1/2)_n\, n!}
\left(-\frac{\beta^{2}}{2\sigma^{2}}\right)^n
\left(\frac{\chi}{\psi}\right)^{n/2}
\frac{K_{\lambda+n+r/2}(\sqrt{\chi\psi})}{K_\lambda(\sqrt{\chi\psi})}.
\]
Combining the powers gives
\[
\left(-\frac{\beta^{2}}{2\sigma^{2}}\right)^n \left(\frac{\chi}{\psi}\right)^{n/2}
=
\left( -\frac{\beta^{2}}{2\sigma^{2}} \sqrt{\frac{\chi}{\psi}} \right)^n.
\]
Finally, inserting this result into (3.8) yields
 \begin{equation}
\begin{aligned}
\mathbb{E}[|X-\mu|^r]
= \;&
\sigma^{r} \frac{2^{r/2}\,\Gamma\!\left(\frac{r+1}{2}\right)}{\sqrt{\pi}}
\left(\frac{\chi}{\psi}\right)^{r/4}
\sum_{n=0}^{\infty}
\frac{(-r/2)_n}{(1/2)_n\, n!}
\left( -\frac{\beta^2}{2\sigma^2} \sqrt{\frac{\chi}{\psi}} \right)^n
\frac{K_{\lambda + n + r/2}(\sqrt{\chi\psi})}{K_\lambda(\sqrt{\chi\psi})},
\end{aligned}
 \end{equation}
which is (3.4).
\end{proof}
\begin{remark}
Since \(\sqrt{\chi/\psi} = \delta/\gamma\) and \(\sqrt{\chi\psi} = \delta\gamma\), equation (3.7) can be rewritten as
\begin{equation}
\mathbb{E}[|X-\mu|^r]
= \;
\sigma^{r} \frac{2^{r/2}\,\Gamma\!\left(\frac{r+1}{2}\right)}{\sqrt{\pi}}
\left(\frac{\delta}{\gamma}\right)^{r/2}
\sum_{n=0}^{\infty}
\frac{(-r/2)_n}{(1/2)_n\, n!}
\left( -\frac{\beta^2 \delta}{2\sigma^2 \gamma} \right)^n
\frac{K_{\lambda + n + r/2}(\delta\gamma)}{K_\lambda(\delta\gamma)}.
\end{equation}
\end{remark}

\begin{remark}[Absolute convergence]
The series
\[
\sum_{n=0}^{\infty}
\frac{(-r/2)_n}{(1/2)_n\, n!}
\left( -\frac{\beta^{2}\delta}{2\sigma^{2}\gamma} \right)^{n}
\frac{K_{\lambda+n+r/2}(\delta\gamma)}{K_{\lambda}(\delta\gamma)},
\qquad \delta=\sqrt{\chi},\;\gamma=\sqrt{\psi},
\]
converges absolutely for all \(r>-\lambda\) and admissible parameters
\(\chi>0,\psi>0\).
This follows from two observations:
(i) the coefficients \(\frac{(-r/2)_n}{(1/2)_n n!}\) are the Maclaurin
coefficients of the entire function \({}_1F_1(-r/2;1/2;\cdot)\);
(ii) the Bessel ratio
\(\frac{K_{\lambda+n+r/2}(\delta\gamma)}{K_{\lambda}(\delta\gamma)}\)
decays exponentially as \(n\to\infty\) by the large-order asymptotics
\(K_\nu(z)\sim\sqrt{\pi/(2\nu)}\,e^{-\nu}\).
Equivalently, convergence is guaranteed by the finiteness of
\(\mathbb{E}[|X|^{r}]\) under the GH law, since
\(\mathbb{E}[W^{s}]<\infty\) iff \(s>-\lambda\) for
\(W\sim\mathrm{GIG}(\lambda,\chi,\psi)\).
\end{remark}

\subsection{Absolute raw moments }
In numerous applications, one is often interested in the absolute moments about the origin, $\mathbb{E}[\,|X|^{r}\,]$, rather than in the central moments.
\begin{theorem} Assume that    $X \sim \mathrm{GH}(\lambda, \alpha, \beta, \delta, \mu, \sigma)$, then, for $r>\max(-1,-\lambda)$, the absolute moment of  $X$ is
\begin{equation}
\begin{aligned}
\mathbb{E}[|X|^r] = \;&
\frac{\sigma^{r}\,2^{r/2}\,\Gamma\!\left(\frac{r+1}{2}\right)}{\sqrt{\pi}}
\sum_{n=0}^{\infty} \frac{(-r/2)_n}{(1/2)_n\, n!} (-1)^n
\sum_{i+j+k=n} \frac{n!}{i!j!k!}
\left(\frac{\mu^2}{2\sigma^2}\right)^i
\left(\frac{\mu\beta}{\sigma^2}\right)^j
\left(\frac{\beta^2}{2\sigma^2}\right)^k \\
&\qquad \times
\left(\frac{\chi}{\psi}\right)^{(r/2 + k - i)/2}
\frac{K_{\lambda + r/2 + k - i}(\sqrt{\chi\psi})}{K_\lambda(\sqrt{\chi\psi})}.
\end{aligned}
\end{equation}
\end{theorem}

\begin{proof}
Recall the mean--variance Gaussian mixture representation
\[
X = \mu + \beta W + \sigma\sqrt{W}\,Z,
\qquad Z\sim\mathcal{N}(0,1),\; Z\perp W.
\]
Hence
\[
|X|^r
= \bigl|\mu + \beta W + \sigma\sqrt{W}\,Z\bigr|^r.
\]
Conditioning on \(W\) and using the law of total expectation,
\begin{equation}
\mathbb{E}\bigl[|X|^r\bigr]
= \mathbb{E}\!\left[
\mathbb{E}\!\left[
\bigl|\mu + \beta W + \sigma\sqrt{W}\,Z\bigr|^r
\;\middle|\; W
\right]
\right].
 \end{equation}
Given \(W=w>0\),  the inner expectation is the absolute moment of a
normal random variable with mean $\mu + \beta w$ and variance $\sigma^{2} w$.
Applying the standard formula (3.6) with
$a = \mu + \beta w$ and scale $\sqrt{w}\,\sigma$ gives
\begin{equation}
\mathbb{E}\!\left[|\mu + \beta w + \sqrt{w}\,Z|^{r}\right]
=
\frac{(\sigma^{2}w)^{r/2}\,2^{r/2}\,\Gamma\!\left(\frac{r+1}{2}\right)}{\sqrt{\pi}}\,
{}_1F_1\!\left(
-\frac{r}{2};\; \frac{1}{2};\;
-\frac{(\mu + \beta w)^{2}}{2\sigma^{2} w}
\right),
\qquad r > -1.
 \end{equation}
Substituting (3.13) into (3.12) yields the unified formula for the origin
absolute moments:
\[
\mathbb{E}\!\left[|X|^{r}\right]
=
\frac{\sigma^{r}\,2^{r/2}\,\Gamma\!\left(\frac{r+1}{2}\right)}{\sqrt{\pi}}\,
\mathbb{E}\!\left[
W^{r/2}\,
{}_1F_1\!\left(
-\frac{r}{2};\; \frac{1}{2};\;
-\frac{(\mu + \beta W)^{2}}{2\sigma^{2} W}
\right)
\right].
\]

Let
\[
A(W) = \frac{(\mu+\beta W)^2}{2\sigma^2 W}
      = \frac{\mu^2}{2\sigma^2}\frac{1}{W} + \frac{\mu\beta}{\sigma^2} + \frac{\beta^2}{2\sigma^2} W.
\]
Expand the confluent hypergeometric function:
\[
{}_1F_1\!\left(-\frac{r}{2};\frac12;-A(W)\right)
= \sum_{n=0}^{\infty} \frac{(-r/2)_n}{(1/2)_n\, n!} \left(-A(W)\right)^n.
\]
Using the trinomial expansion,
\[
\left( \frac{\mu^2}{2\sigma^2}\frac{1}{W} + \frac{\mu\beta}{\sigma^2} + \frac{\beta^2}{2\sigma^2} W \right)^n
= \sum_{\substack{i,j,k \ge 0 \\ i+j+k=n}}
\frac{n!}{i!j!k!}
\left(\frac{\mu^2}{2\sigma^2}\right)^i
\left(\frac{\mu\beta}{\sigma^2}\right)^j
\left(\frac{\beta^2}{2\sigma^2}\right)^k
W^{-i} W^{k}.
\]
Thus
\begin{equation}
\begin{aligned}
&\mathbb{E}\!\left[ W^{r/2} {}_1F_1\!\left(-\frac{r}{2};\frac12;-A(W)\right) \right] \\
&= \sum_{n=0}^{\infty} \frac{(-r/2)_n}{(1/2)_n\, n!} (-1)^n
\sum_{i+j+k=n} \frac{n!}{i!j!k!}
\left(\frac{\mu^2}{2\sigma^2}\right)^i
\left(\frac{\mu\beta}{\sigma^2}\right)^j
\left(\frac{\beta^2}{2\sigma^2}\right)^k
\mathbb{E}\!\left[ W^{r/2 + k - i} \right].
\end{aligned}
\end{equation}
Bying using (3.3),
\[
 \mathbb{E}(W^{k-i+r/2})
= \left(\frac{\chi}{\psi}\right)^{(k-i)/2 + r/4}
\frac{K_{\lambda+k-i+r/2}(\sqrt{\chi\psi})}{K_\lambda(\sqrt{\chi\psi})}.
\]
Finally, inserting this result into (3.14) yields (3.11).
\end{proof}

\begin{remark} The   series
\[
\sum_{n=0}^{\infty}
\frac{(-r/2)_n}{(1/2)_n\,n!}(-1)^n
\sum_{i+j+k=n}\frac{n!}{i!\,j!\,k!}
\left(\frac{\mu^{2}}{2\sigma^{2}}\right)^{i}
\left(\frac{\mu\beta}{\sigma^{2}}\right)^{j}
\left(\frac{\beta^{2}}{2\sigma^{2}}\right)^{k}
\left(\frac{\chi}{\psi}\right)^{\frac{r/2+k-i}{2}}
\frac{K_{\lambda+r/2+k-i}(\sqrt{\chi\psi})}
     {K_{\lambda}(\sqrt{\chi\psi})}
\]
converges absolutely for every \(r>-\lambda\) and for all admissible
parameters \((\mu,\beta,\sigma,\chi,\psi)\) with \(\chi>0,\psi>0\).
This follows from three facts:

 1. The outer coefficient
\(\displaystyle \frac{(-r/2)_n}{(1/2)_n n!}(-1)^n\) is the Maclaurin
coefficient of the entire confluent hypergeometric function
\({}_1F_1(-r/2;1/2;\cdot)\); hence the outer series has infinite radius
of convergence.

2. The inner multinomial sum is a finite polynomial in \(n\) and thus
introduces no new divergence.

3. For fixed \(\lambda\) and \(z=\sqrt{\chi\psi}>0\), the Bessel ratio
\(\displaystyle K_{\lambda+r/2+k-i}(z)/K_{\lambda}(z)\) decays
exponentially as \(|k-i|\to\infty\) by the large-order asymptotics
\(K_\nu(z)\sim \sqrt{\pi/(2\nu)}\,e^{-\nu}(1+o(1))\).
\end{remark}

 \section{Special cases}

 \subsection{Absolute moments of the NIG distribution}
Let \(X\sim\mathrm{NIG}(\alpha,\beta,\delta,\mu;\sigma)\) with density (2.3).
Then, for \(r>-1\),
\[
\mathbb{E}\bigl[|X-\mu|^r\bigr]
=
\delta^{\,r}\sigma^{\,r}\,
\frac{K_{\frac{r+1}{2}}(\delta\gamma)}
{K_{1}(\delta\gamma)}\,
\left(\frac{2}{\gamma\sigma}\right)^{\!r/2}
\frac{\Gamma\!\left(\frac{r+1}{2}\right)}{\sqrt{\pi}};
\]
 and
for \(r>-1/2\),
\[
\mathbb{E}[|X|^r]
=
\sum_{n=0}^{\infty}
\frac{(-r/2)_n}{(1/2)_n\,n!}
\left(-\frac{\beta^2\delta}{2\alpha\sigma}\right)^{\!n}
\frac{K_{\frac12+n+\frac{r}{2}}(\delta\alpha)}
{K_{\frac12}(\delta\alpha)}\,
{}_1F_1\!\left(-n;\tfrac12;\frac{\mu\beta}{\sigma^2}\right).
\]
The series converges absolutely for all admissible parameters.

  \subsection{Absolute moments of the hyperbolic distribution}
Let \(X\) follow the hyperbolic distribution with density (2.4).
 Then,
for \(r>-2\),
\[
\mathbb{E}\bigl[|X-\mu|^r\bigr]
=
\frac{2^{r/2}\Gamma\!\left(\frac{r+1}{2}\right)}{\sqrt{\pi}}\,
(\delta\sigma)^{r}\,
\left(\frac{1}{\gamma\sigma}\right)^{\!r}
\sum_{n=0}^{\infty}
\frac{(-r/2)_n}{(1/2)_n\,n!}
\left(-\frac{\beta^2\delta}{2\gamma}\right)^{\!n}
\frac{K_{1+\frac{r}{2}+n}(\delta\gamma)}
{K_{1}(\delta\gamma)}.
\]
In the symmetric case \(\beta=0\), this reduces to
\[
\mathbb{E}\bigl[|X-\mu|^r\bigr]
=
\frac{2^{r/2}\Gamma\!\left(\frac{r+1}{2}\right)}{\sqrt{\pi}}\,
(\delta\sigma)^{r}\,
\left(\frac{1}{\gamma\sigma}\right)^{\!r}
\frac{K_{1+\frac{r}{2}}(\delta\gamma)}
{K_{1}(\delta\gamma)}.
\]
For \(r>-2\),
\[
\mathbb{E}[|X|^r]
=
e^{\beta\mu/\sigma^2}
\sum_{n=0}^{\infty}
\frac{(-r/2)_n}{(1/2)_n\,n!}
\left(-\frac{\beta^2\delta}{2\gamma}\right)^{\!n}
\frac{K_{1+\frac{r}{2}+n}(\delta\gamma)}
{K_{1}(\delta\gamma)}
\,L_n^{\left(-\frac12\right)}\!\left(-\frac{\mu\beta}{\sigma^2}\right),
\]
where \(L_n^{(\alpha)}(\cdot)\) denotes the generalized Laguerre polynomial.
Equivalently,
\[
\mathbb{E}[|X|^r]
=
\sum_{n=0}^{\infty}
\frac{(-r/2)_n}{(1/2)_n\,n!}
\left(-\frac{\beta^2\delta}{2\gamma}\right)^{\!n}
\frac{K_{1+\frac{r}{2}+n}(\delta\gamma)}
{K_{1}(\delta\gamma)}
\,{}_1F_1\!\left(-n;\tfrac12;\frac{\mu\beta}{\sigma^2}\right).
\]

 \subsection{Absolute moments of the symmetric GH distribution}

Setting \(\beta=0\) in the GH model yields the symmetric GH law
\(X=\mu+\sigma\sqrt{W}Z\), where \(W\sim\mathrm{GIG}(\lambda,\delta^2,\alpha^2)\).
The absolute central moments are given by
\[
\mathbb{E}\bigl[|X-\mu|^r\bigr]
=
\frac{2^{r/2}\Gamma\!\left(\frac{r+1}{2}\right)}{\sqrt{\pi}}\,
\sigma^{r}\,
\left(\frac{\delta}{\alpha}\right)^{\!r}
\frac{K_{\lambda+\frac{r}{2}}(\delta\alpha)}
{K_{\lambda}(\delta\alpha)},
\qquad r>-2\lambda,
\]
 the absolute raw moments are given by
\[
\mathbb{E}[|X|^r]
=
\frac{2^{r/2}\Gamma\!\left(\frac{r+1}{2}\right)}{\sqrt{\pi}}\,
\sigma^{r}
\sum_{n=0}^{\infty}
\frac{(-r/2)_n}{(1/2)_n\,n!}
\left(-\frac{\mu^2}{2\sigma^2}\right)^{\!n}
\left(\frac{\delta}{\alpha}\right)^{\!r-2n}
\frac{K_{\lambda+\frac{r}{2}-n}(\delta\alpha)}
{K_{\lambda}(\delta\alpha)},  \qquad r>-2\lambda.
\]
In the fully symmetric case $\mu=0$, this reduces to
\[
\mathbb{E}[|X|^r]
=
\frac{2^{r/2}\Gamma\!\left(\frac{r+1}{2}\right)}{\sqrt{\pi}}\,
(\sigma\delta/\alpha)^{r}
\frac{K_{\lambda+\frac{r}{2}}(\delta\alpha)}
{K_{\lambda}(\delta\alpha)}.
\]

\subsection{Absolute moments of the VG distribution}

For \(\lambda>0\), letting \(\delta\to0\) reduces the GH distribution to
the variance-gamma (VG) distribution (2.8).

For any real number $r > 0$, the $r$-th absolute central moment is given by the closed-form expression:
\begin{equation}
\begin{aligned}
\mathbb{E}\left[|X-\mu|^r\right]
= \;&
\frac{2^{\,r+\lambda+\frac{1}{2}}\,\Gamma(r+1)\,\Gamma(r+2\lambda)}
{\Gamma(r+\lambda+1)}
\, \gamma^{-2r-2\lambda} \, \sigma^{-2\lambda} \\
&\times \left[
\left( \sqrt{\gamma^2\sigma^2 + \beta^2} + \beta \right)^{2\lambda-1}
+
\left( \sqrt{\gamma^2\sigma^2 + \beta^2} - \beta \right)^{2\lambda-1}
\right],
\end{aligned}
\end{equation}
where $\Gamma(\cdot)$ denotes the Gamma function.

If the distribution is symmetric, i.e., $\beta = 0$, the formula (4.1) simplifies to
\begin{equation*}
\label{eq:symmetric}
\mathbb{E}\left[|X-\mu|^r\right]
=
\frac{2^{\,r+\lambda+\frac{3}{2}}\,\Gamma(r+1)\,\Gamma(r+2\lambda)}
{\Gamma(r+\lambda+1)}
\, \gamma^{-2r-1} \, \sigma^{-1}.
\end{equation*}

 The absolute raw moments are given by
 \begin{equation}
\mathbb{E}[|X|^r]
=
e^{\beta\mu/\sigma^2}
\sum_{n=0}^{\infty}
\frac{(-r/2)_n}{(1/2)_n\,n!}
\left(-\frac{\beta^2}{\alpha^2}\right)^{\!n}
\frac{\Gamma\!\left(\lambda+\frac{r}{2}+n\right)}
{\Gamma(\lambda)}\,
\left(\frac{2}{\alpha^2}\right)^{\!r/2+n}
\,L_n^{\left(-\frac12\right)}\!\left(-\frac{\mu\beta}{\sigma^2}\right),
 \end{equation}
where \(L_n^{(\alpha)}(\cdot)\) denotes the generalized Laguerre polynomial.
Equivalently,
 \begin{equation*}
\mathbb{E}[|X|^r]
=
\sum_{n=0}^{\infty}
\frac{(-r/2)_n}{(1/2)_n\,n!}
\left(-\frac{\beta^2}{\alpha^2}\right)^{\!n}
\frac{\Gamma\!\left(\lambda+\frac{r}{2}+n\right)}
{\Gamma(\lambda)}\,
\left(\frac{2}{\alpha^2}\right)^{\!r/2+n}
\,{}_1F_1\!\left(-n;\tfrac12;\frac{\mu\beta}{\sigma^2}\right).
 \end{equation*}
Gaunt (2025) obtained the result  for $r=2k$ with $k\in\mathbb{Z}^+$.

In the special case where the location parameter is zero (\(\mu = 0\)),  it simplifies to a closed form involving the Gauss hypergeometric function \({}_2F_1\) (see, Gaunt (2023)):
\begin{equation*}
\mathbb{E}|X|^r =
\frac{2^{r/2} \Gamma\!\left(\frac{r+1}{2}\right) \Gamma\!\left(\lambda + \frac{r}{2}\right)}
{\sqrt{\pi} \, \Gamma(\lambda)}
\, \sigma^r \, \gamma^{-r} \,
{}_2F_1\!\left(
-\frac{r}{2},\; \lambda + \frac{r}{2};\; \frac12;\;
-\frac{\beta^2}{2\sigma^2 \gamma^2}
\right)
\end{equation*}
for \(r > 0\).

 \subsection{Absolute moments of the Skewed $t$ Distribution}

Let \(X\) follow the skewed $t$ distribution with density,
corresponding to the GH limit \(\lambda=-\nu/2\), \(\gamma\to0\).
Let \(W\) follow an inverse-gamma distribution with shape \(\nu/2\) and scale
\(\delta^2\beta^2/\sigma^2\), and write \(X=\mu+\beta W+\sigma\sqrt{W}Z\). For $0<r<\nu$,
the absolute central moments are given by
\[
\mathbb{E}\bigl[|X-\mu|^r\bigr]
=
\frac{2^{r/2}\Gamma\!\left(\frac{r+1}{2}\right)}{\sqrt{\pi}}\,
\sigma^{r}\,
\frac{\Gamma\!\left(\frac{\nu-r}{2}\right)}
{\Gamma\!\left(\frac{\nu}{2}\right)}\,
\left(\frac{2|\beta|}{\delta^2}\right)^{\!r/2}
\,{}_1F_1\!\left(-\frac{r}{2};\frac{\nu-r}{2};-\frac{\beta^2}{\delta^2|\beta|}\right),
\]
 and the absolute raw moments are given by
\[
\mathbb{E}[|X|^r]
=
e^{\beta\mu/\sigma^2}
\sum_{n=0}^{\infty}
\frac{(-r/2)_n}{(1/2)_n\,n!}
\left(-\frac{\beta^2}{\delta^2|\beta|}\right)^{\!n}
\frac{\Gamma\!\left(\frac{\nu-r}{2}-n\right)}
{\Gamma\!\left(\frac{\nu}{2}\right)}\,
\left(\frac{2|\beta|}{\delta^2}\right)^{\!r/2+n}
\,L_n^{\left(-\frac12\right)}\!\left(-\frac{\mu\beta}{\sigma^2}\right),
\]
where \(L_n^{(\alpha)}(\cdot)\) denotes the generalized Laguerre polynomial.
Equivalently,
\[
\mathbb{E}[|X|^r]
=
\sum_{n=0}^{\infty}
\frac{(-r/2)_n}{(1/2)_n\,n!}
\left(-\frac{\beta^2}{\delta^2|\beta|}\right)^{\!n}
\frac{\Gamma\!\left(\frac{\nu-r}{2}-n\right)}
{\Gamma\!\left(\frac{\nu}{2}\right)}\,
\left(\frac{2|\beta|}{\delta^2}\right)^{\!r/2+n}
\,{}_1F_1\!\left(-n;\tfrac12;\frac{\mu\beta}{\sigma^2}\right).
\]
If the location parameter \(\mu = 0\), then \({}_1F_1(-n; \tfrac12; 0) = 1\), and the series collapses to a single non-zero term. The moment is given by the classical formula:
\begin{equation*}
\mathbb{E}[|X|^r]
=
\sigma^r
\frac{2^{r/2} \,\Gamma\!\left(\frac{r+1}{2}\right) \Gamma\!\left(\frac{\nu-r}{2}\right)}
{\sqrt{\pi} \;\Gamma\!\left(\frac{\nu}{2}\right)}
\left(\frac{\nu}{2}\right)^{-r/2}.
\end{equation*}

 \section{Conclusion}

We have demonstrated that the variance-mean mixture representation provides a powerful and conceptually straightforward framework for deriving  absolute moments of the GH distribution. This approach reduces what would otherwise require complex integrals of Bessel functions
to simple Gaussian moment calculations and the well-known moments of the GIG distribution. The
derivation for the VG distribution follows effortlessly via a limiting argument.
Our presentation significantly simplifies the existing literature, making these important formulas
transparent, easily verifiable, and straightforward to implement in practice. It would be interesting to investigate  absolute moments (e.g., moments of norms) for multivariate GH distribution, with direct applications in portfolio risk management. We leave this as an interesting project for future
studies.

\noindent{\bf Acknowledgements}

 This research was supported by the National Natural Science Foundation of China (No. 12071251).

{\bf Data availability}

No data was used for the research described in the article.\\

{\bf Appendix}.  Non-trivial zeros of the Riemann zeta function.  The Riemann zeta function is defined by
\begin{equation*}
\zeta(s)=\frac{1}{(1-2^{1-s})\Gamma(s+1)}\int_{0}^{\infty} \frac{t^{s}\,e^{t}}{(e^{t}+1)^{2}}\,dt,
\qquad 0<\Re(s)<1.
\end{equation*}
 For $s=a+bi$, $0<a<1, b\in\mathbb{R}$, let
$$I(a,b):=I(s)=\int_{0}^{\infty} \frac{t^{s}\,e^{t}}{(e^{t}+1)^{2}}\,dt,$$
which can be written as
\begin{eqnarray*}
I(a,b)&=&\int_{0}^{\infty} \frac{t^{a}\,e^{t}}{(e^{t}+1)^{2}}(\cos(b\ln t)+i\sin(b\ln t))\,dt\\
&=&\int_{0}^{\infty} \frac{t^{a}\,e^{t}}{(e^{t}+1)^{2}}\cos(b\ln t)\,dt\\
&&+i\int_{0}^{\infty} \frac{t^{a}\,e^{t}}{(e^{t}+1)^{2}}\sin(b\ln t)\,dt\\
&&\equiv J_1(a,b)+iJ_2(a,b).
\end{eqnarray*}
 It can be proved that $I(s) = 0$ if and only if $\zeta(s) = 0$, which is equivalent to $\zeta(1-s) = 0$ by the Riemann functional equation
$$\zeta(s) = 2^s \pi^{s-1} \sin\left(\frac{\pi s}{2}\right) \Gamma(1-s) \zeta(1-s),$$
  and in turn to $I(1-s) = 0$.
It follows that $\zeta(s)=0$ if and only if the following four equations hold
simultaneously:
\[
\begin{cases}
J_1(a,b)=0, \\
J_2(a,b) = 0, \\
J_1(1-a,b)=0,\\
J_2(1-a,b)=0.
\end{cases}
\]
 For sufficiently large $|b|$, it can be proved that
\[
\begin{cases}
J_1(a,b) < J_1(1-a,b), & 0<a<\frac{1}{2},\\[4pt]
J_1(a,b) = J_1(1-a,b), & a=\frac{1}{2},\\[4pt]
J_1(a,b) > J_1(1-a,b), & \frac{1}{2}<a<1.
\end{cases}
\]
and
\[
\begin{cases}
J_2(a,b)>J_2(1-a,b), & 0<a<\dfrac12,\\[2mm]
J_2(a,b)=J_2(1-a,b), & a=\dfrac12,\\[2mm]
J_2(a,b)<J_2(1-a,b), & \dfrac12<a<1.
\end{cases}
\]
Observe that $\zeta(s) = 0$ requires $J_1(a,b) = J_1(1-a,b)$ and $J_2(a,b) = J_2(1-a,b)$ whenever $0 < a < 1$.  We conclude that every non-trivial zero of $\zeta(s)$ has real part exactly $\frac{1}{2}$.

\end{document}